\documentclass[10 pt]{amsart}
\usepackage{epsfig}
\usepackage{latexsym,amssymb,amsmath}
\usepackage{fullpage}
\usepackage[applemac]{inputenc}
\usepackage{graphicx}

\usepackage{eufrak}

\newtheorem{theorem}{Theorem}
\newtheorem{proposition}{Proposition}
\newtheorem{corollary}{Corollary}

\newcommand{\Keywords}[1]{\par\noindent 
{\small{\em Keywords\/}: #1}}

\def\Bbb{\mathbb}
\def\a{\alpha}

\def\d{\delta}
\def\D{\Delta}
\def\l{\lambda}

\def\l{\lambda}
\def\m{\mu}

\def\s{\sigma}

\title{A Product of integer partitions }

\author{Alain Goupil $^{1}$  } 
\address{Département de mathématiques et d'informatique\\ Université du Québec à Trois-Rivières\\
c.p. 500, Trois-Rivières, G9A 5H7} 
\email{alain.goupil@uqtr.ca}

\begin{document}

\footnotetext[1]{Work partially supported by a grant from NSERC.}

\begin{abstract}
            I present a bijection on integer partitions that leads to recursive expressions,  closed formulae and generating functions for the  cardinality  of certain sets of partitions of  a positive integer $n$.  The bijection leads also to a product on partitions that is associative with a natural grading thus defining a free associative algebra on the set of integer partitions. As an outcome of the computations, certain sets of integers appear that I call  difference sets and the product of the  integers in a difference set is an invariant for a family of sets of partitions.  The main combinatorial objects used in these constructions are the central  hooks of the Ferrers diagrams of partitions. 
\Keywords{partition, product, monoid, enumeration, linear representation.}  
\end{abstract}
\maketitle

\section{Introduction}
\label{intro}
A primitive strategy that lies at the hart of enumerative combinatorics is the  old Pythagorean school   idea of  using geometric objects  to study arithmetics. This connection between geometry and numbers dates back to the fourth century B.C. and one  famous illustration of  this duality  is the well known expansion of perfect squares as sums of odd integers. The immediate visual proof  of that identity is provided  by the following figure that breaks a square of dots into disjoint hooks of odd size:

\noindent
\begin{center}
\begin{tabular}{cc}

\begin{minipage}[t] {90 mm}\vspace{.0cm}
\vspace{.3 cm}
$n^2=1+3+\cdots +2n-1\qquad\qquad \longleftrightarrow$
\end{minipage}
&\begin{minipage}[t]{20 mm}\vspace{.0cm}
\includegraphics{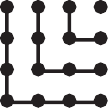}\\

\end{minipage}
\\
\end{tabular}
\end{center}
The hooks  in the preceding  figure were called {\it gnomons} by the Ancient greeks and were used in their geometric number theory  to provide recursive rules and identities on integers.  A first idea that is developped in this paper is to mimic this type of manipulation on sets of partitions by using  the same hooks.  A second idea that has guided this work is to develop an algebraic structure on the set $P$ of partitions by introducing a non commutative product on $P$. These two ideas are used to count partitions. 

The paper is organized as follow. In the remaining of section \ref{sec1}, the necessary notations are established. In section \ref{sec2} I present  the bijections and numerical results needed to define in section \ref{sec3} a product of partitions and describe the derived algebraic structures. In particular two linear representations as triangular matrices of dimension $3$ of the monoïd of partitions are presented. Section \ref{sec4} is devoted to enumeration of certain classes of partitions. In particular, an exact non recursive formula for  the computation of $p(n)$ is given.  In section \ref{sec5}, I give a generating function in several variables for the numbers $p((k_1,\ldots , k_r))$ of $r$-hooks with hook type $(k_1,\ldots , k_r)$. In section \ref{sec6}, I discuss invariance and prove a property of sets of partitions 
that share the same set of differences. I conclude in section \ref{sec7} with a number of remarks and questions arising from this work. 

 \subsection{Notation} \label{sec1}Most conventions and definitions needed in this paper are  borrowed from \cite{A} and \cite{AE} and are outlined in the following paragraphs.
 
 A {\em partition}  $\l=(\l_1, \ldots, \l_m)$  of a positive integer $n$ is a weakly decreasing sequence $\l_1\geq \l_2\geq \ldots \geq\l_m$ of positive integers such that $\sum_{i=1}^m\l_i=n$. 
  The number  $n$ is called the {\em weight} of $\l$ and we write $\l\vdash n$. I shall also use the multiplicative notation for partitions and write $\l=1^{\ell_1}2^{\ell_2}\ldots n^{\ell_n}$ to mean that the number $i$ appears $\ell_i$ times in the partition $\l$. The Ferrers diagram of a partition $\l$ is, in the cartesian convention,  an array of $m$ rows of cells ordered from bottom to top in increasing order of size such that the $i$th  row contains $\l_i$ cells and these Ferrers diagrams will also be denoted $\l$. The conjugate partition of a partition $\l$ is denoted $\l'$.  A {\em hook} is a partition with multiplicative form $\l=1^kn-k$ with $0\leq k\leq n-1$ and its Ferrers diagram is made  of one column of cells, possibly empty,  sitting on the leftmost  cell  of a unique row of cells which I will call the {\em corner} of the hook.   
 I will be interested with the {\it central hooks} of a partition $\l$ i.e. hooks whose corner lie on the main diagonal of the Ferrers diagram of $\l$ so that any integer partition can be seen as a diagonal superposition of  these hooks. When I will need to diagonally superpose $r$ hooks to obtain a partition $\l$, I will say that $\l$ is a $r${\em -hook}. $1$-hooks will sometimes be called hooks. Thus the Durfee square of a r-hook has size $r\times r$.  For example, in  figure \ref{fig1},  the partition $\l=(4,4,2,1))$ on the left is a $2$-hook and the partition $\m=(5,4,3,2,2)$ on the right is a $3$-hook. The innermost hook and outermost hook of a Ferrers diagram $\l$ will be called respectively  the {\it inner hook} and  the {\it outer hook} of $\l$  . 
 
 \begin{figure}[htbp]
\begin{center}
\includegraphics{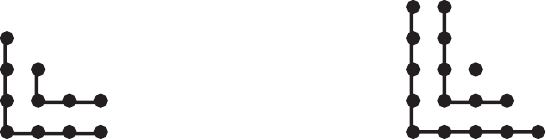}
\caption{a) A $2$-hook \hspace{2cm} b) A  $3$-hook\hspace{1cm} }
\label{fig1}
\end{center}
\end{figure}

Let us denote by $P$ the set of all partitions, $P(n)$  the set of partitions of $n$ and by  $p(n)$ the cardinality of $P(n)$.  Also denote by $P(n,r)$  the set of partitions of $n$ that are $r$-hooks and by $p(n,r)$ their number which is also the number of partitions of $n$ with Durfee square of size $r\times r$. We will also need the set $PI(n,r,k)$ of $r$-hooks of weight $n$ with inner hook of weight $k$ and its corresponding cardinality $pi(n,r,k)$ and the set $PO(n,r,k)$ of $r$-hooks of weight $n$ with outer hook of weight $k$ and its cardinality $po(n,r,k)$.  To any partition $\l$ which is a $r$-hook, we associate the $r$-tuple $(k_1,k_2,\ldots, k_r)$ of lengths of each central hook of $\l$ starting from the left of the Ferrers diagram and we will call it the {\it hook type} of $\l$. 
Thus  $P((k_1,k_2,\ldots , k_r))$ will denote the set of partitions with  hook type $(k_1,k_2,\ldots , k_r)$ and $p((k_1,k_2,\ldots , k_r))$ will be the cardinality of that set. The numbers $p((k_1,k_2,\ldots , k_r))$ can also be seen as the number of partitions of $k_1+k_2+\ldots + k_r$ with $k_i\geq k_{i+1}+2$ for all $1\leq i <r$ (see \cite{G} or partitions with $2$-distinct part in \cite{AE}).

Observe that $r$-hooks are partitionning the set $P(n)$ of partitions when $1\leq r\leq \lfloor \sqrt{n}\rfloor$: 
\begin{equation}
\label{eq0}
p(n)=\sum_{r=1}^{\lfloor \sqrt{n}\rfloor} p(n,r)
\end{equation}
and that any set of $r$-hooks of weight $n$ can itself be partitionned according to hook types:
\begin{equation}
\label{eq00}
p(n,r)=\sum_{\begin{subarray} {c}(k_1,k_2,\ldots , k_r)\\ \sum_{i=1}^r k_i=n, 
\, k_i\geq k_{i+1}+2
\end{subarray}}
p((k_1,k_2,\ldots , k_r))
\end{equation}
If a partition is a $r$-hook of weight $n$ with inner hook of weight $k$ then $n\geq r^2+r(k-1)$. Finally if $f(x)$ is any formal series, then $f(x)|_{x^n}$ will denote  the coefficient of $x^n$ in $f(x)$. 

\section{Bijections}\label{sec2}
\begin{theorem}
\label{th1} For positive  integers $k, n$,
 there is a bijection between the cartesian product  $P(n,r) \times P(k,1)$  and the set $PI(n+(k+1)r+k,r+1,k)$ of $(r+1)$-hooks of weight $n+(k+1)r+k$ with inner hook of weight $k$.
\end{theorem}
\begin{proof} The following colored figure \ref{fig2}  provides a  proof  without words. 
\begin{figure}[htbp]
\begin{center}
\includegraphics{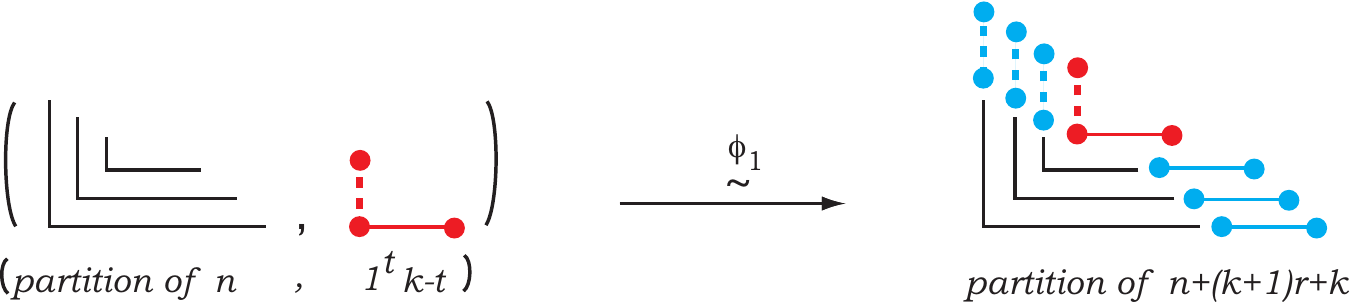}
\caption{A bijection between $P(n,r)\times P(k,1)$ and $P(n+(k+1)r+k),r+1,k)$}
\label{fig2}
\end{center}
\end{figure}
The function $\phi_1$ takes a couple $(\l,h)$ made of a  $r$-hook $\l$ and a $1$-hook $h=1^tk-t$,  places the hook $h$ inside the partition $\l$ as illustrated and adds to each hook of $\l$ a vertical part of size $t+1$ and a horizontal part of size $k-t$. When the initial partition is a $r$-hook of weight $n$, the resulting partition is a $(r+1)$-hook of weight $n+(k+1)r+k$ because we add $k+1$ dots to each of the $r$ hooks of $\l\vdash n$ and the hook $h$ of size $k$.The map $\phi_1$ is invertible since it is always possible to remove an inner hook  from  any partition along with the corresponding vertical and horizontal strips from each of its non inner hooks. 
\end{proof}
\begin{corollary}\label{corol1} For positive integers $k_1,\ldots , k_r$,
the number  $p((k_1,k_2,\ldots , k_r))$ of partitions with  hook type $(k_1,k_2,\ldots , k_r)$ is given by 
\begin{equation}\label{eq1}
p((k_1,k_2,\ldots , k_r))=k_r(k_{r-1}-k_r-1)(k_{r-2}-k_{r-1}-1)\cdots(k_1-k_2-1)
\end{equation}
\end{corollary}
\begin{proof}
We have to observe first  that the number of $1$-hooks  of length $k$ is equal to $k$. In other words
 \begin{equation*}
p(n,1)=n\qquad \forall n\geq 1.  
\end{equation*}
 Then, using the bijection in theorem \ref{th1}, we obtain the following  recurrence  on the number of central hooks:
\begin{equation}\label{eq11}
p((k_1,k_2,\ldots , k_r))=k_rp((k_1-(k_r+1),k_2-(k_r+1),\ldots , k_{r-1}-(k_r+1)))
\end{equation}
Then by induction hypothesis, we obtain (\ref{eq1}).
\end{proof}
A visual proof of corollary \ref{corol1} is obtained by observing that a difference $(k_{i-1}-k_i-1)$
between the sizes of two consecutive hooks  can be seen as  a {\it degree of freedom } of the hook of size $k_{i-1}$ with respect to the hook of size $k_i$. 

\begin{figure}[htbp]
\begin{center}
\includegraphics{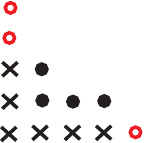}
\caption{The free positions on a hook}
\label{fig3}
\end{center}
\end{figure}
\noindent
For instance, in  figure \ref{fig3}, the positions with empty circles on the first hook are free and the positions with crosses are not.  Thus one way to obtain all partitions with hook type $(k_1,k_2,\ldots , k_r)$ is to start by sliding the inner hook in $k_r$ different positions, then for each position of the inner hook, we can slide the neighbour  hook  in a number of positions equal to its number of circles plus one and so on for each hook. Thus the multiplicative property of the set $P((k_1,k_2,\ldots , k_r))$ given by equation (\ref{eq1}) appears naturally. 

Special instances of corollary \ref{corol1}  are:
\begin{eqnarray}\label{du1}
p((2n-1,\ldots,3,1))&=&1\\ \label{pn2}
p(n,2)&=&\sum_{k=1}^{\lfloor \frac n 2\rfloor -1}k(n-2k-1)=\frac{\binom {n}3}{4} -\frac{(n-1)}8   
(n \mod 2)\\  \label{pn3}
p(n,3)&=&\sum_{k_3=1}^{\lfloor \frac n 3\rfloor -2 }\sum_{k_2=k_3+2}^{\lfloor \frac {n-k_3} 2\rfloor -1 }k_3(k_2-k_3-1)(n-2k_2-k_3-1)\\
&=& \sum_{d_3=1}^{\lfloor \frac n 3\rfloor -2 }\frac {d_3}4\binom{n-3d_3-2}{3}-\frac{d_3}8(n-3(d_3+1))((n+d_3)\mod 2)
\end{eqnarray}
We will develop an exact expression for $p(n,3)$ in section \ref{sec5}. The sequence of differences $k_r, \,k_{r-1}-k_r-1,\, k_{r-2}-k_{r-1}-1,\, \cdots k_1-k_2-1$ appearing in corollary \ref{corol1} will keep appearing in the rest of the paper and we will call  it the {\it  difference sequence} of the corresponding hook shape and its set of partitions. The associated set  $\{k_r, \,k_{r-1}-k_r-1,\, k_{r-2}-k_{r-1}-1,\, \cdots k_1-k_2-1\}$ will be called  the {\it difference set} associated with the hook type $(k_1,k_2,\ldots , k_r)$. Any hook type $(k_1,k_2,\ldots, k_{t})$ and its weight $n$  are recovered from the difference sequence 
$(d_1,d_2,\ldots, d_{t})$ as follow:
\begin{eqnarray}
(k_1,k_2,\ldots ,k_r)&=&
 (r-1+\sum_{i=1}^rd_i, r-2+\sum_{i=2}^rd_i, \ldots , d_{r-1}+d_r+1,d_r)\\
 \label{eqn}
 n&=&\sum_{i=1}^r id_i+\binom{r}{2}
\end{eqnarray}

The number $n$ given in (\ref{eqn}) wil be called the weight of $\d=(d_1,d_2,\ldots, d_{t})$ and denoted $|\d|$.
The bijection described in theorem \ref{th1} may be extended to a bijection from  pairs of partitions  to partitions as in the next corollary. Recall that for positive integers $n,r,k$,  $PO(n,r,k)$  is the set of $r$-hooks of weight $n$ with outer hook of weight $k$. 

\begin{corollary}  \label{cor3}
For any positive integers $n,m,r_1$ and $r_2$, there is a bijection $\phi_2$ between the cartesian product $P(n,r_1)\times PO(m,r_2,k)$ and $P(n+m+r_1(k+1),r_1+r_2)$.
\end{corollary}
\begin{proof}
We apply a map $\phi_2$ similar to $\phi_1$ in theorem \ref{th1}. As shown in figure \ref{fig4}, any pair of partitions $(\l, \m) \in P(n,r_1)\times PO(m,r_2,k)$ is transformed into a partition where $u$ is placed inside $\l$ as if it was a hook. Then to each hook of $\l$,  we add on top one column of size equal to the first column of $\m$  and on the right, one row of size equal  to the size of the first row of $\m$. We thus obtain a partition of weight $n+m+r_1(k+1)$ with $r_1+r_2$ hooks. It is immediate that this transformation is bijective because we can dissect back any partition  that is a $(r_1+r_2)$-hook by first removing the inner $r_2$-hook $\m$ and then erasing from the remaining $r_1$-hook  the columns and rows of sizes equal to the sizes of the column and row of the outer hook of $\m$.

\begin{figure}[htbp]
\begin{center}
\includegraphics{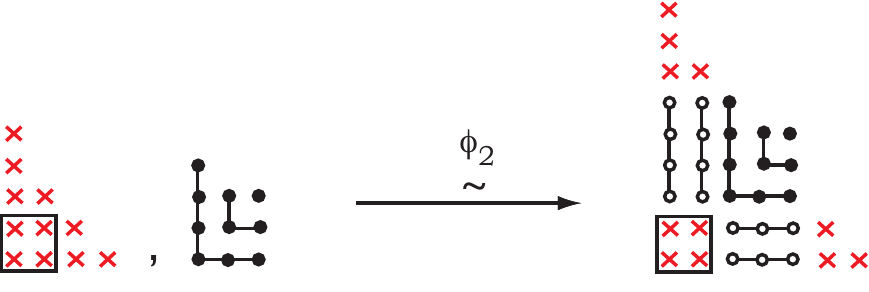}
\caption{A bijection between $P(n,r_1)\times PO(m,r_2,k)$
and $P(n+m+r_1(k+1),r_1+r_2)$}
\label{fig4}
\end{center}
\end{figure}
\end{proof}

Observe that we can also extend the bijection in corollary \ref{cor3} to a map from $m$-tuples of partitions to partitions.  In particular, we can map $m$-tuples of $1$-hooks to $m$-hooks.
\begin{corollary}
\label{cor2}
 For positive integers $d_1,d_2.\ldots, d_m$, there is a bijection from the cartesian product
$P(d_1,1)\times P(d_2,1)\times \cdots \times P(d_m,1)$ to the set of $m$-hooks with hook type $(k_1,k_2,\ldots , k_m)$ satisfying:
\begin{equation*}
k_1=r-1+\sum_{i=1}^rd_i,\; k_2= r-2+\sum_{i=2}^rd_i, \ldots , \; k_m=d_m
\end{equation*}
\end{corollary}
\begin{proof}
This is an immediate consequence of a sequential application of the bijection in  theorem \ref{th1} from left to right on the sequence of hooks and of the fact that inserting hooks in this manner is an associative operation. The inverse bijection consists in the peeling of hooks from a $m$-hook , reducing the $m$-hook  to a sequence of hooks whose size are the terms in the difference sequence of the $m$-hook.  The next figure shows an example of this bijection applied to a sequence of three hooks. 
\begin{figure}[htbp]
\begin{center}
\includegraphics{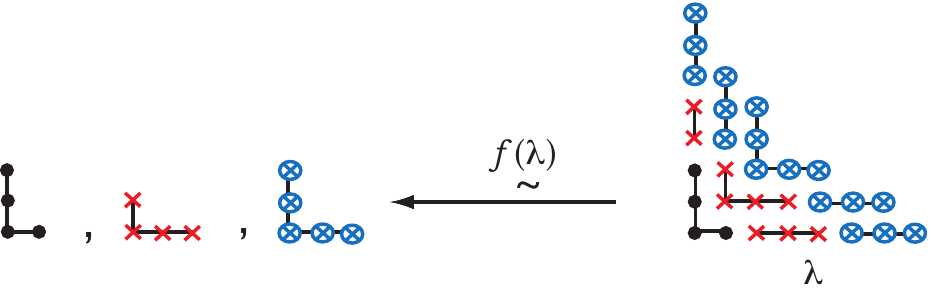}
\caption{A bijection between $P(d_1,1)\times P(d_2,1)\times P(d_3,1)$ and $P((2+d_1+d_2+d_3,1+d_2+d_3, d_3))$}
\label{fig5}
\end{center}
\end{figure}
\end{proof}

A noticeable  difference between the bijection in  corollary \ref{cor2} and the simple action of directly inserting hooks one inside another  without adding  vertical and horizontal segments  is that in corollary \ref{cor2},  there is no constraint on the  size of each hook involved, which is not the case in an ordinary insertion of hooks.  Thus the correspondance between $r$-tuples of hooks and partitions in corollary \ref{cor2} can be used as a definition for a unique non commutative  {\it hook factorisation} of a partition into a "product", or sequence, of hooks. Let us call $f(\l)$ this hook factorisation of a partition $\l$.

\section{ A free monoid structure}\label{sec3}

The bijection between Ferrers diagrams $\l,\m$ described in corollary \ref{cor3} and illustrated in figure \ref{fig4} gives birth to an associative, non commutative  binary operation that we will  call the {\it product of the partitions} $\l$ and $\m$ and denote it $\l*\m$.  The hook factorisation $f(\l)$ of a Ferrers diagram $\l$ into a non commutative product of  hooks as illustrated in figure \ref{fig5} is the inverse of the product and it satisfies 
\begin{equation}
\label{eq66}
f(\l*\m)=f(\l)* f(\m)
\end{equation}

Since hook factorisation is unique, the pair $(P,*)$  forms a free associative, non commutative monoid with the set of hooks as its set of indecomposable generators.   Thus hook factorization of Ferrers diagrams is analogous to prime factorization of integers.  There is a natural length function on $P$ given by the size of  the Durfee square of the partitions $\l\in P$ that we denote $du(\l)$. We thus have
\begin{equation}\label{eq7}
du(\l*\m)=du(\l)+du(\m)
\end{equation}
so that the map $du:(P,*)\rightarrow (N,+)$ is a monoid homomorphism.  Conjugation of partition, denoted $\l'$,  is compatible with product:
\begin{equation*}
(\l*\m)'=\l'*\m'
\end{equation*} 
and a square $n\times n$ is factored as a product of $n$ single dots as in figure \ref{fig6}.
\begin{figure}[htbp]
\begin{center}
\includegraphics{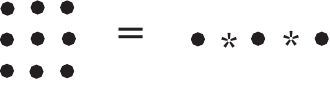}
\caption{Factorization of a square}
\label{fig6}
\end{center}
\end{figure}

 \noindent
 {\bf Quotient monoids}.  Let us define an equivalence relation $\sim$ on the set  $P$ of partitions as follow: $\l \sim\m\iff \l$ and $\m$ have the same hook type.  It is easy to verify that $\sim$ is stable under right and left multiplication i.e. $\sim$ is a congruence on $(P,*)$. That is, for every $\l,\m,\a\in P, \l \sim\m\Rightarrow \l*\a \sim\m*\a$ and $ \a*\l \sim\a*\m$.  
 
 The set $P/\sim $ of congruence classes of $P$ is  also a monoid called the quotient of $P$ by $\sim$  and we have the choice among several indexing sets for the elements of $P/\sim$.  Let us consider the following three sets: the set  $H$ of hook types, the set $\D$ of differences sequences and the sets $P_{\leq r}(n-r^2)$  of partitions of $n-r^2$ with at most $r$ parts for all $1\leq r\leq \sqrt(n)$. It is known that these three sets have the same cardinality and we give in the following proposition  the bijective maps  between these three sets.
 \begin{proposition} \label{prop1}

 Let \\
 $H(n)=\{(k_1,\ldots,k_r):k_i\geq k_{i+1}+2,1\leq i < r,r\leq\sqrt{n} \}$ be the set of all hook types of $n$.  \\ $\D(n)=\{(d_1,\ldots,d_r):\sum_{i=1}^r id_i+\binom r 2 = n, r\leq\sqrt{n} \}$ be the set of difference sequences of $n$.\\
 $\Pi (n)=\{ \m\vdash n-r^2,\ell(\m)\leq r,r \leq\sqrt{n} \}$ be the set of partitions of $n-r^2$ with at most $r$ parts. 
 
 \noindent
 We have  $card (H(n))=card (\D(n))=card (\Pi (n))$.
\end{proposition} 
\begin{proof}
There are six maps presented in  table \ref{indx}, two for each index set and it is easy to prove that these are all bijections. We illustrate these bijections when $n=13$ and the reader may verify that the product of the entries in each row of the middle column add to $p(13)=101$.
\end{proof}

\begin{table}[htdp]
\caption{Three index sets of $P(13)/\sim $}
\begin{center}
\begin{tabular}{|c|c|c|c|}
\hline
H(n) & $\D(n)$& $\Pi(n)$&\\
$ (k_1,\ldots , k_r)$&$ \scriptstyle(k_1-k_2-1,\ldots ,k_{r-1}- k_r-1,k_r)$&$\scriptstyle (k_1-(2r-1),k_2-(2r-3),\ldots , k_r-1)$&r\\
$ \scriptstyle(-1+\sum_{i=1}^r(d_i+1),-1+\sum_{i=2}^r(d_i+1),\ldots,d_r)$&$(d_1,\ldots , d_r)$&$ \scriptstyle(\sum_{i=1}^r(d_i-1),\sum_{i=2}^r(d_i-1),\ldots,d_r-1)$& \\
$ \scriptstyle(\m_1+2r-1,\m_2+2r-3,\ldots , \m_r+1)$&$ \scriptstyle(\m_1-\m_2+1,\ldots ,\m_{r-1}- \m_r+1,\m_r+1)$&$(\m_1,\ldots , \m_r)$&\\
\hline
(13)&(13)&(12)&1\\
(12,1)&(10,1)&(9,0)&2\\
(11,2)&(8,2)&(8,1)&2\\
(10,3)&(6,3)&(7,2)&2\\
(9,4)&(4,4)&(6,3)&2\\
(8,5)&(2,5)&(5,4)&2\\
(9,3,1)&(5,1,1)&(4,0,0)&3\\
(8,4,1)&(3,2,1)&(3,1,0)&3\\
(7,5,1)&(1,3,1)&(2,2,0)&3\\
(7,4,2)&(2,1,2)&(2,1,1)&3\\
\hline
\end{tabular}
\end{center}
\label{indx}
\end{table}

Observe that a consequence of proposition \ref{prop1} is that the partitions $\m\in P_{= r}(n-r^2)$ with exactly $r$ parts are in bijection with the set of hook types $(k_1,k_2,\ldots , k_r)$ of partitions of weight $n$ with last 
 term $k_r\geq 2$. 
We can also use the first Rogers-Ramanujan Identity to obtain a fourth  index set:  the number of partitions of $n$ with parts congruent to $\pm 1 \mod 5$ (see \cite{A} or \cite{AE}).  The product $\circ$ in $P/\sim$ inherited from $(P,*)$ behaves as follow on these three index sets:
 
 \begin{eqnarray}
{\rm Hook \,types}&&:(k_1,\ldots , k_{r})\circ  (k_{1}',\ldots , k_{s}')= 
(k_1+k_{1}'+1,\ldots ,k_r+k_{1}'+1, k_{1}',\ldots , k_{s}')\\
\label{dif1}
{\rm Difference \,sequences}&&:(d_1,\ldots , d_{r})\circ  (d_{1}',\ldots , d_{s}')= 
(d_1,\ldots , d_{r}, d_{1}',\ldots , d_{s}')\\
{\rm Partitions \,of}\,  n-r^2&&:(\m_1,\ldots , \m_{r})\circ  (\l_{1},\ldots , \l_{s})= 
(\m_1+\l_1,\ldots , \m_{r}+\l_1, \l_{1},\ldots , \l_{s})
\end{eqnarray}

I choose the set $\D$ of difference sequences as the indexing set for $P/\sim$ because of (\ref{dif1}) and corollary \ref{corol1} and also because  the operation of taking the cardinality of the subsets  of $P$  indexed with difference sequences is a monoid homomorphism from $(P/\sim,\circ)$ to the monoid $(\mathbb{N},\cdot)$ of natural numbers with ordinary multiplication:
\begin{equation*}
card: (P/\sim,\circ)\: \rightarrow \: (\mathbb{N},\cdot), card(\d_1\circ \d_2)=card(\d_1)\cdot card(\d_2),
\quad \forall \d_1,\d_2 \in \D
\end{equation*}

The subset {\it DU} $\in P$ of all Durfee squares  is closed under multiplication $*$ and any Durfee square is alone in its congruence class so that  {\it DU} = {\it DU}$/\sim$ is a submonoid both of {\it P} and of {\it P/}$\sim$ . Only Durfee squares are sent to $1$ by the morphism {\it card} so that the set {\it DU} is  the kernel of {\it card}: {\it ker(card):=DU}.  The set {\it HOOK} of partitions that have no Durfee square as right factor is also a submonoid of $( P,*)$ and its corresponding set of congruence classes {\it HOOK/}$\sim$ is a submonoid of {\it P/}$\sim$. Both  monoids $(P,*)$  anf $(P/ \sim,\circ)$ are thus direct products of hooks with Durfee squares:

\begin{eqnarray*}
P&=& \mbox{ \it HOOK} \otimes \mbox{ \it DU}\\
P/\sim &=& \mbox{ \it HOOK}/\sim \otimes \mbox{ \it DU}/\sim. 
\end{eqnarray*}

The sets {\it HOOK} and  {\it HOOK/}$\sim$ may also be seen as the quotients {\it P}$/${\it DU} and  $(P/\sim)/\mbox{ \it DU}$ respectively.
We will consider the enumeration of these sets of partitions in the next section. 

Finally, there exists a linear representation of partitions $\l\in (P,*)$ as  $3\times 3$ lower triangular matrices over the set $\Bbb N$ of natural numbers. 
\begin{proposition} \label{prop11} For any partition $\lambda$
let $o(\l)$ be the size of the outer  hook of $\l$. The map $\phi_3: (P,*)\rightarrow (\Bbb N^{3\times 3},\cdot)$ given by
\begin{equation}
\phi_3 (\l)=\left(  \begin{array}{ccc}
		1&0&0\\
		o(\l)+1&1 &0\\
		|\l|& du(\l)&1	
 	\end{array}
\right)
\end{equation}
is a monoid homomorphism.
\end{proposition}
\begin{proof} On one hand we have equation (\ref{eq7}) and
\begin{eqnarray*}
o(\l*\m)&=&o(\l)+o(\m)+1\\
|\l*\m|&=&|\l|+|\m|+du(\l)(o(\m)+1)\\
\end{eqnarray*}

On the other hand matrix multiplication gives
\begin{equation*}
\left(  \begin{array}{ccc}
		1&0&0\\
		o(\l)+1&1 &0\\
		|\l|& du(\l)&1	
 	\end{array}
\right)
\left(  \begin{array}{ccc}
		1&0&0\\
		o(\m)+1&1 &0\\
		|\l|& du(\m)&1	
 	\end{array}
\right)
=
\left(  \begin{array}{ccc}
		1&0&0\\
		o(\l)+o(\m)+2&1 &0\\
		|\l|+|\m|+du(\l)(o(\m)+1)& du(\l)+du(\m)&1	
 	\end{array}
\right)
\end{equation*}
so that $\phi_3(\l*\m)=\phi_3(\l)\cdot \phi_3(\m)$.
\end{proof}
There is a version of this linear representation which can be defined on the monoid $(P/\sim,\circ)$ and which sends differences sequences to upper triangular matrices. 

\begin{proposition}   The map 
$\phi_4: (P/\sim,\circ)\rightarrow (\Bbb N^{3\times 3},\cdot)$ given by
\begin{equation}
\phi_4 (\delta)=\left(  \begin{array}{ccc}
		1&1&\sum_i id_i\\
		0&1 &\sum_i d_i\\
		0& 0&1	
 	\end{array}
\right)
\end{equation}
where $\delta=(d_1,\ldots , d_r)$, is a monoid homomorphism i.e. $\phi_4(\delta_1\circ \delta_2)= \phi_4(\delta_1)\cdot \phi_4(\delta_2)$ for all $\d_1,\d_2\in \D$.
\end{proposition} 
\begin{proof} The proof is similar to that of proposition \ref{prop11} and is omitted.
\end{proof} 

\section{Counting partitions}\label{sec4}
The recurrence appearing in equation (\ref{eq11}) can be extended to the following recurrence:
\begin{corollary} For integers $n$ and $r$ with  $ 2 \leq r\leq \lfloor \sqrt{n}\rfloor$, the number of $r$-hooks of weight $n$ satisfies the following recurrence:
\begin{equation}
\label{eq111}
p(n,r)= \sum_{k=1}^{\lfloor \frac{n-r(r-1)}r\rfloor}kp(n-(k+1)r+1,r-1)
\end{equation}
\end{corollary}
Using corollary  \ref{corol1} we can obtain a closed form expression for the number of partitions of $n$ that are $r$-hooks. We only have to count the number of admissible $r$-tuples $(k_1,k_2,\ldots , k_r)$ of weight $n$ i.e. the $r$-tuples that are hook types of a partition of $n$. A $r$-tuple  of integers $(k_1,k_2,\ldots , k_r)$ is the hook type of a partition if and only if $k_1\geq k_2+2$, $ k_2\geq k_3+2, \ldots$ $k_{r-1}\geq k_r+2$ so that we have 
\begin{equation}\label{eq2}
p(n,r)=\sum_{\begin{subarray}{c} (k_1,k_2,\ldots , k_r)\\ k_i\geq k_{i+1}+2\\ \sum_i k_i=n \end{subarray}}
k_r(k_{r-1}-k_r-1)(k_{r-2}-k_{r-1}-1)\cdots(k_1-k_2-1)
\end{equation}
The preceding expression can be refined so that the values of the $k_i$ are bounded:
\begin{theorem} For positive integers $n,r$ such that $n\geq r^2$, the number $p(n,r)$  is given by  
\begin{equation}
\label{eq3}
p(n,r)=\sum_{k_r=1}^{\lfloor \frac nr\rfloor-r+1}\sum_{k_{r-1}=k_r+2}^{\lfloor \frac{n-k_r}{r-1}\rfloor-r+2}
\cdots \sum_{k_{2}=k_3+2}^{\lfloor \frac{n-\sum_{i=3}^r k_i}2\rfloor-1}
k_r(k_{r-1}-k_r-1)(k_{r-2}-k_{r-1}-1)\cdots (k_{1}-k_2-1)
\end{equation}
\end{theorem}
\begin{proof} Equation \ref{eq3} is a straightforward consequence of equation (\ref{eq2})  and corollary \ref{corol1}.
\end{proof}
A consequence of (\ref{eq0}) and (\ref{eq3}) is: 
\begin{corollary} For all positive  integers $n$ we have:
\begin{equation}\label{eq5}
p(n)=\sum_{r=1}^{\lfloor \sqrt{n}\rfloor}
\sum_{k_r=1}^{\lfloor \frac nr\rfloor-r+1}\sum_{k_{r-1}=k_r+2}^{\lfloor \frac{n-k_r}{r-1}\rfloor-r+2}
\cdots \sum_{k_{2}=k_3+2}^{\lfloor \frac{n-\sum_{i=3}^r k_i}2\rfloor-1}
k_r(k_{r-1}-k_r-1)(k_{r-2}-k_{r-1}-1)\cdots (k_{1}-k_2-1)
\end{equation}
\end{corollary}


Compared to the known recursive formulae, equation (\ref{eq5}) is not  efficient to compute the values $p(n)$ and the  question of improving its efficiency  arises.  I will propose a number of solutions in what follows. One direction that can be explored to improve equation (\ref{eq5}) is to collect several sets $P(k_1,\ldots,k_r)$ and count them simultaneously.  This can be done when we observe that the difference sequence of a Durfee square  is $(1,1,\ldots , 1)$ so that the partitions of weight $n$ that have a Durfee square of size say $s-1$ as right factor and a $1$-hook as left factor have difference sequence of the form $(n-s^2+1,1^{s-1})$ and we obtain a closed expression for their number:

\begin{proposition} Let $H_1(n)$ be the set of  partitions of $n$ that are the product of a $1$-hook with a Durfee square of  arbitrary size, possibly zero,  on its right and let   $h_1 (n)$   be the cardinality of that set. We have
\begin{equation}\label{h1}
h_1(n)=\lfloor \sqrt n\rfloor(n+1)-\frac {\lfloor \sqrt n\rfloor(\lfloor \sqrt n\rfloor+1)(2\lfloor \sqrt n\rfloor+1)}6
\end{equation}
\end{proposition} 
\begin{proof}
These partitions have a Durfee square as right factor of possible size $0,1,\ldots , \lfloor \sqrt n\rfloor$ so that 
\begin{equation}\label{h2}
h_1(n)=(n-1^2+1)+(n-2^2+1)+ \cdots (n-\lfloor \sqrt n\rfloor^2+1)
\end{equation}
and (\ref{h1}) is an immediate consequence of (\ref{h2}). 
\end{proof}

Now let us consider the sets $H_r(n)=(HOOK_r\times DU)\cap P(n)$ of partitions of $n$  that are the product of a $r$-hook, $r\geq 2$, with inner hook of size at least $2$, a non degenerate hook, with a Durfee square of arbitrary size, possibly zero, on its right. Let $h_r(n)$ be the cardinality of $H_r(n)$. Since the factorization of a partition as a product of a $r$-hook with a Durfee square  is unique, the sets $H_1(n),H_2(n), \ldots, H_r(n)$ are disjoints and form a partition of the set $P(n)$. Thus we have
\begin{eqnarray} \label{eqpn}
\nonumber
P(n)&=&H_1(n)\cup H_2(n)\cup \ldots \cup H_r(n),\\
\Rightarrow p(n)&=&h_1(n)+h_2(n)+\ldots h_r(n), \; r=\lfloor\frac{n}{\lfloor \sqrt n\rfloor+1}\rfloor
\end{eqnarray}
\begin{proposition} The cardinality of the set $H_2(n)$ is given by the expression 

\begin{equation}\label{h2d}
h_2(n)=\sum_{j=2}^{\lfloor \sqrt{n-2}\rfloor}\frac {\binom{n-j^2+4}3}4-\frac{(n-j^2+3)((j+n)\mod 2)}8-(n-j^2+1),\quad n\geq 6
\end{equation}
\end{proposition} 
\begin{proof}
The set $H_2(n)$ is the disjoint union of the products of a non degenerate $2$-hook with a Durfee square of maximal size $\lfloor \sqrt{n-2}\rfloor-2 $:
\begin{equation*}
H_2(n)= \big (\bigcup_{i=0}^{\lfloor \sqrt{n-2}\rfloor-2} HOOK_2\times Du_i\big)\bigcap P(n)
\end{equation*}
where $Du_i$ is the Durfee square of size $i\times i$. Since $(HOOK_2\times Du_i)(n)=P(n-(i+2)^2+4,2)-P(n-(i+2)^2+4,2,1)$ i.e. the partitions of  $n$ which are 2-hook times a Durfee square of size $i$ are in bijection with the $2$-hooks of weight $n-(i+2)^2+4$ minus those with inner hook of size $1$. This last observation and identity (\ref{pn2}) imply (\ref{h2d}).
\end{proof}

To obtain closed expressions for $h_i(n),i\geq 3$, we need closed expressions for $p(n,i)$  but  
let us partition the sets  $H_r(n)$ according to their difference sequences  so that their cardinalities satisfy:
\begin{equation}\label{eq23}
h_r(n)=\sum_{d_r=2}^{\lfloor \frac n r \rfloor-2}\sum_{d_{r-1}=1}^{\lfloor \frac {n-rd_r} {r-1} \rfloor-2}
 \cdots \sum_{d_2=1}^{\lfloor \frac {n-\sum_{i=3}^rid_i} 2 \rfloor-2}
 h_r(n,d_2,d_3,\ldots , d_r)
\end{equation}
where $h_r(n,d_2,d_3,\ldots , d_r)$ is the number of partitions of $n$ with difference sequence of the form \\
$(n-(r+s)^2-\sum_{i=2}^r id_i+\binom{r+1}{2},d_2,\ldots , d_r,1^s)$. The numbers $h_r(n,d_2,d_3,\ldots , d_r)$ satisfy the following recurrences
\begin{eqnarray*}
h_r(n,d_2,d_3,\ldots , d_r)&=&(d_2\cdots d_r) \; h_r(n-\sum_{i=2}^r i(d_i-1),1^{r-1})\\
h_r(n,1^r) &=& h_1(n)-n-(n-3)-(n-8)\cdots -(n-r^2+1)\\
&=& h_1(n)-rn+\frac{r(2r+5)(r-1)}6
\end{eqnarray*}
so that equation (\ref{eqpn}) becomes 
\begin{eqnarray}
\nonumber
p(n)=h_1(n)+&&\sum_{r=2}^{\lfloor\frac{n}{\lfloor \sqrt n\rfloor+1}\rfloor}\sum_{d_r=2}^{\lfloor \frac n r \rfloor-2}
\sum_{d_{r-1}=1}^{\lfloor \frac {n-rd_r}{r-1} \rfloor-2} \cdots \sum_{d_2=1}^{\lfloor n -\sum_{i=3}^r id_i \rfloor-2}\\
&&\prod_{i=2}^r d_i 
\left[  h_1(n-\sum_{i=2}^r  i(d_i-1))- (r-1)(n-\sum_{i=2}^r i(d_i-1))+\frac{(r-1)(2r+3)(r-2)}6 \right]
\end{eqnarray}
where $h_1(n)$ is given by (\ref{h1}).

 \subsection{Durfee $\times $ HOOK}
Let us count the partitions of $n$ that are the product of a Durfee square of size $r\times r$ on the left with a $1$-hook. As illustrated in figure \ref{fig7}, we have 
\begin{equation}
n=r^2+rj+ri+i+j-1=r^2+r(k+1)+k=(r+1)(k+r)
\end{equation}
where $k=i+j-1$ is the size of the $1$-hook.   The number $ dh(n)$ of these partitions is given by the factorisations of $n$ as a product of two integers:
\begin{equation}
 dh(n)=\sum_{{\begin{subarray}{c} x|n \\ x\leq \sqrt(n)  \end{subarray}}}
 (\frac n x -x+1)
\end{equation}

\begin{figure}[htbp]
\begin{center}
\includegraphics{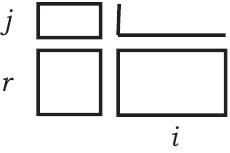}
\caption{The product of a  Durfee square with a $1$-hook}
\label{fig7}
\end{center}
\end{figure}
The numbers $dh(n)$ thus have a number theoretic interpretation as the sum over the pairs of divisors of $n$ of the positive differences of the divisors  plus one and their generating function is:
\begin{equation}
\sum_{n\geq 1}dh(n)x^n=\sum_{r\geq 1}  \frac{x^{r^2}}{(1-x^r)^2} 
\end{equation}


\section{Generating functions} \label{sec5}Starting with Leonard Euler, Generating functions have been used extensively in the study of partitions. The list of contributors is too important  and my knowledge  too small to present a proper account of their work.  The goal of this section is to present a generating function  in several variables of the number of $r$-hooks  which I have not seen in the literature on partitions and from which I derive an exact expression for $p(n,3)$.

It is well known (see \cite{A}) that the generating function for the number $p(n,r)$ of partitions of $n$ that are $r$-hooks is 
\begin{equation}\label{eq31}
\sum_{n\geq 0}p(n,r)x^n=\frac{x^{r^2}}{\prod_{i=1}^r(1-x^i)^2}
\end{equation}
This generating function may be refined into a series in several variables: 

\begin{proposition}\label{th3}
\begin{align}\label{eq9}
\sum_{r\geq 1} \sum_{(k_1,k_2,\ldots ,k_r)}p((k_1,k_2,\ldots ,k_r))x_1^{k_1}x_2^{k_2}\cdots x_r^{k_r}&=&
\sum_{r\geq 1} x_1^{r}x_2^{r-1}\cdots x_r^{1}\prod_{i=1}^r\frac d{dx_i} \left(\frac 1{1-x_1x_2\cdots x_i}
\right)\\
&=&
\label{eq91}
\sum_{r\geq 1} x_1^{2r-1}x_2^{2r-3}\cdots x_r^{1}\prod_{i=1}^r \left(\frac 1{1-x_1x_2\cdots x_i}\right)^2
\end{align}
\end{proposition}
\begin{proof}
We have to show that the coefficient of $x_1^{k_1}x_2^{k_2}\cdots x_r^{k_r}$ in the right hand side of equation (\ref{eq9}) is $p((k_1,k_2,\ldots ,k_r))=(k_1-k_2-1)(k_2-k_3-1)\cdots (k_{r-1}-k_r-1)k_r$. First observe that
\begin{align*}
\frac d{dx_i} \left(\frac 1{1-x_1\cdots x_i}\right)
&=\frac{x_1\cdots x_{i-1}}{(1-x_1\cdots x_i)^2}\\
\nonumber
&= x_1\cdots x_{i-1}\left[ 1+2 x_1\cdots x_i +3(x_1\cdots x_i)^2+\cdots k(x_1\cdots x_i)^{k-1}+\ldots \right]
\end{align*}
Thus
\begin{align*}
x_1^{r}x_2^{r-1}\cdots x_r^{1}\prod_{i=1}^r\frac d{dx_i} \left(\frac 1{1-x_1\cdots x_i}\right)
=& x_1^{r}x_2^{r-1}\cdots x_r^{1}\left[  1+2x_1+\ldots +kx_1^{k-1}+\ldots \right]\times \\
\nonumber 
&x_1\left[  1+2x_1x_2+\ldots +k(x_1x_2)^{k-1}+\ldots \right] \times \\
\nonumber 
&x_1x_2\left[  1+2x_1x_2x_3+\ldots +k(x_1x_2x_3)^{k-1}+\cdots \right] \times \\
\nonumber 
&\vdots \\
\nonumber 
&x_1\cdots x_{r-1}\left[  1+2x_1\cdots x_r+\ldots +k(x_1\cdots x_r)^{k-1}+\ldots \right] \\
\nonumber 
= x_1^{2r-1}x_2^{2r-3}\cdots x_{r-1}^{3}&x_r\left[ \sum_{(d_1,d_2,\ldots ,d_r)} d_1d_2\cdots d_r  x_1^{d_1-1}(x_1x_2)^{d_2-1}\cdots (x_1\cdots x_r)^{d_r-1}
\right]\\
\nonumber 
= \sum_{(d_1,d_2,\ldots ,d_r)} d_1d_2\cdots d_r  &x_1^{2r-1+\sum_{i=1}^r(d_i-1)}
x_2^{2r-3+\sum_{i=2}^r(d_i-1)}\cdots  
x_{r-1}^{3+\sum_{i=r-1}^r+(d_i-1)}
x_r^{d_r-1}
\end{align*}
Recalling that for any sequence  $(d_1,d_2,\ldots ,d_r)$ of differences, the corresponding hook type  $(k_1,k_2,\ldots ,k_r)$ is given by $(k_1,k_2,\ldots ,k_r)=
 (r-1+\sum_{i=1}^r(d_i), r-2+\sum_{i=2}^r(d_i), \ldots , d_{r-1}+d_r+1,d_r)$ which is the sequence of exponents in the last expression and proves the theorem.
\end{proof}

The following well known expressions for $p(n,r)$ and $p(n)$ are  specializations of theorem \ref{th3} that are easily derived:

\begin{corollary}
\begin{align}
\label{eq112}
p(n,r)=& \frac{x^{\binom{r+1}2}}{r!}\prod_{i=1}^r \frac d {dx}\left( \frac 1 {1-x^i}\right)\Big{|}_{x^n}\\
\label{eq12}
	=&\,  x^{r^2}\prod _{i=1}^r \frac 1 {(1-x^i)^2}\Big{|}_{x^n}
\end{align}
\begin{align}
\label{eq13}
p(n)=&\sum_{r=1}^{\lfloor \sqrt{n}\rfloor} \frac{x^{\binom{r+1}2}}{r!}\prod_{i=1}^r \frac d {dx}\left( \frac 1 {1-x^i}\right)\Big{|}_{x^n}\\
\label{eq14}
	=&\,  \sum_{r=1}^{\lfloor \sqrt{n}\rfloor} x^{r^2}\prod _{i=1}^r \frac 1 {(1-x^i)^2}\Big{|}_{x^n}
\end{align}
\end{corollary}

\begin{proof}
Equations (\ref{eq112}) to (\ref{eq14})  are immediately obtained from the substitutions $x_i\rightarrow x$ in (\ref{eq9}) and (\ref{eq91}) and from equations (\ref{eq0})  and  (\ref{eq00}). These equations appear in equivalent form in \cite{A}(section 2.2)  and  \cite{G}  and equation  (\ref{eq12}) is attributed to Euler. 
\end{proof}

We can derive from equation (\ref{eq31}) exact expressions for $p(n,r)$, $r=2,3,\ldots$.

\begin{proposition}\label{prop7}
\begin{eqnarray}\label{eq38a}
p(n,3)&=&\frac {(n+1)(n-5)}{12960}(3n^3-33n^2+83n-13)-
\frac {(n-3)}{32}(\lfloor \frac n 2\rfloor-\lfloor \frac {n-1} 2\rfloor)\\
\nonumber
&&+\lfloor \frac {n} 3\rfloor \frac {(3n-8)}{81}
-\lfloor \frac {n-1} 3\rfloor \frac {(3n-10)}{81}
-\frac {2}{81}\lfloor \frac {n-2} 3\rfloor \\ \label{eq39a}
&=&\frac {(n+1)(n-5)}{12960}(3n^3-33n^2+83n-13)+
\begin{cases}
\frac {(3n-8)}{81}-\frac {(n-3)}{32} &  \text {if $n\equiv 0 \mod 6$}\\
\frac {2}{81} &  \text {if $n\equiv 1\mod 6$}\\
-\frac {(n-3)}{32} &  \text {if $n\equiv 2 \mod 6$}\\
\frac {(3n-8)}{81} &  \text {if $n\equiv 3 \mod 6$}\\
\frac {2}{81}-\frac {(n-3)}{32} &  \text {if $n\equiv 4 \mod 6$}\\
0 & \text {if $n\equiv 5\mod 6$}
\end{cases}
\end{eqnarray}
\end{proposition}
\begin{proof}
For $r=3$, expand  equation (\ref{eq31})  as a partial fraction and rearrange the expression to our needs:
\begin{eqnarray}
\nonumber
\sum_{n\geq 0}p(n,3)x^n&=&\frac{x^9}{(1-x)^2(1-x^2)^2(1-x^3)^2}\\
\nonumber
&=& \frac{1}{36(1-x)^6}- \frac{1}{6(1-x)^5}+ \frac{169}{432(1-x)^4}
- \frac{185}{432(1-x)^3}+ \frac{307}{1728(1-x)^2}+ \frac{11}{432(1-x)}\\
\nonumber
&&-\frac{3+x}{27(1+x+x^2)}+\frac{1+x}{27(1+x+x^2)^2}+\frac{1}{16(1+x)}-\frac{1}{64(1+x)^2}\\
\nonumber
&=& \frac{1}{36(1-x)^6}- \frac{1}{6(1-x)^5}+ \frac{169}{432(1-x)^4}
- \frac{185}{432(1-x)^3}\\
\nonumber
&& + \frac{1}{8(1-x^2)}-\frac{(4-x)}{27(1-x^3)}+\frac{(x+2)(1+x^2+x^3)}{27(1-x^3)^2}-\frac{(1+x^2)}{32(1-x^2)^2}+\frac{1}{8(1-x)^2}\\
&=&
\label{eq38}\sum_{n\geq 0}\left[ \frac{\binom{n+5}{5}}{36}-\frac{\binom{n+4}{4}}{6}+\frac{169\binom{n+3}{3}}{432}-\frac{185\binom{n+2}{2}}{432}+
\frac{(n+1)}{8}
\right]x^n\\
\nonumber
&& +\left[ \frac{1}{8}+\frac{x(n+1)}{16}
\right]x^{2n}
+\left[ \frac{x+ (x+2)(1+x^2+x^3)(n+1)}{27}
\right]x^{3n}
\end{eqnarray}
Now to reduce the coefficients of $x^{2n}$ and $x^{3n}$ in the last expression into coefficients of $x^n$ , we need two  functions $f_1,f_2:\Bbb{N}\rightarrow \Bbb{N}$ such that 
\begin{equation*}
f_1(n)=
\begin{cases}
1 &  \text {if $n \equiv 0 \mod 2$}\\
0 & \text {otherwise}
\end{cases},\qquad
f_2(n)=
\begin{cases}
1 &  \text {if $n\equiv 0 \mod 3$ }\\
0 & \text {otherwise}
\end{cases}
\end{equation*}
Several descriptions satisfy these conditions  and I choose 
$f_1(n)=\lfloor\frac{n}2  \rfloor - \lfloor\frac{n-1}2  \rfloor$, $f_2(n)=\lfloor\frac{n}3  \rfloor - \lfloor\frac{n-1}3  \rfloor$ so that from equation (\ref{eq38}) we obtain
\begin{eqnarray*}
p(n,3)&=&\left[ \frac{\binom{n+5}{5}}{36}-\frac{\binom{n+4}{4}}{6}+\frac{169\binom{n+3}{3}}{432}-\frac{185\binom{n+2}{2}}{432}+
\frac{(n+1)}{8}\right]+ \frac{1}{8}\left(\lfloor\frac{n}2  \rfloor - \lfloor\frac{n-1}2  \rfloor\right)\\
&&+\frac{(n+1)/2}{16}
\left(\lfloor\frac{n-1}2  \rfloor - \lfloor\frac{n-2}2  \rfloor\right)
+\frac{(5n-6)}{81}\left(\lfloor\frac{n}3  \rfloor - \lfloor\frac{n-1}3  \rfloor\right)\\
&&+\frac{2(n+2)}{81}\left(\lfloor\frac{n-1}3  \rfloor - \lfloor\frac{n-2}3  \rfloor\right)
+\frac{2(n+1)}{81}\left(\lfloor\frac{n-2}3  \rfloor - \lfloor\frac{n-3}3  \rfloor\right)
\end{eqnarray*}
from which we deduce equation (\ref{eq38a}). Equation (\ref{eq39a}) is obtained directly from (\ref{eq38a}) and the definitions of $f_1$ and $f_2$.
\end{proof}
The expression for $p(n,2)$ in equation (\ref{pn2}) is obtained similarly. 

\section{Invariance under the difference product }\label{sec6}
As we have seen in corollary \ref{corol1}, the  set of partitions $\l$ with hook type $(k_1,k_2,\ldots , k_r)$ has cardinality given by the product of corresponding differences $d=\prod_{i=1}^r d_i$. Conversely, the product $d$ can be factored as a product of integers, including the factor $1$, in an infinite number of ways, each way indexing a distinct set of partitions. The product of  differences  $\prod_{i=1}^r d_i$ is thus invariant under multiplication with Durfee squares and becomes an invariant for an infinite  family of sets of partitions all having the same difference product and distinct difference sequences. 
This stability property occurs elsewhere in combinatorics. In particular  Frobenius (\cite{Fr}) and others observed it for conjugacy classes of permutations in the symmmetric group $S_n$  and  irreducible representations  of $S_n$, which are both indexed with integer partitions. If we let $H_i=HOOK_i\otimes DU$ be the set of non degenerate $i$-hooks multiplied on the right with an arbitrary Durfee square, we have the decomposition 
\begin{equation*}
P= DU \cup H_1\cup H_2 \cup \cdots 
\end{equation*}

A given product $d$ may also be obtained from difference sequences  of  identical weights. For example the difference sequences $(4,1,1,2), (2,4,1,1),(2,2,1,2),(8,1,1,1)$ all have same weight $n= 23$ and same product. 
There is a number of  natural questions that are simply jumping at us  regarding this last observation. First, is it possible to give a  characterization  of the sets of partitions with given product of differences and same weight?  In other words: {\it Describe the transformations on partitions tha leave a given difference product and weight invariants}.  In particular, how many difference sequences give the same weight and product? For example it is easy to prove that for all odd $n>6$ and positive integers $a\neq b$ with $a+b=(n-1)/2$, there are precisely two pairs  of differences namely $(2a,b),(2b,a)$ with same weight and product so that  $p(n,2)$ is (n-3) times the average area of all integer rectangles of perimeter $n-1$.
%


A difference product $d$ is left invariant by any permutation of the set $(d_1,d_2,\ldots , d_r)$ but what is the effect on the weight $n$   of permuting the entries in a difference sequence?  
\begin{proposition}\label{th4} 
Let $D= \{ d_1\geq d_2\geq \ldots  \geq d_r\}$ be a difference set and let $S(D)=\{ (d_{\s_1},\ldots , d_{\s_r}),\s\in S_r\}$ be the set of difference sequences obtained from $D$. Then  
$\d_1=(d_1, d_2, \ldots  ,d_r)$ is the only difference sequence in $S(D)$ with minimum weight  and 
$\d_2=(d_r, d_{r-1}, \ldots  ,d_1)$ is the only difference sequence  in $S(D)$ with maximum weight.
Moreover the difference between the maximum and minimum weight is $|\d_2|-|\d_1|=\sum_{i<j}(d_j-d_i)$.
\end{proposition}

\begin{proof}
We prove the first statement. If $\d'\neq \d_1$ then there is a position $1\leq i<r$ with a rise in $\d'$ i.e. such that $d_i'< d_{i+1}'$. Exchanging $d_i'$ and $d_{i+1}'$ in $\d'$ decreases the weight i.e. $|\tau_i\d_1'|<|\d_1'|$ if $\tau_i=(i,i+1)$ is the adjacent transposition. Repeating this process properly, we obtain $\d_1$ by strictly decreasing the weight at each step. This proves the first statement. The maximality of $\d_2$ is proved similarly. 

The proof of the third statement follows from the argument in the proof of the first statement.  Starting from  $\d_1$, if we exchange two adjacent differences $d_i,d_j$ when $d_i> d_j$,  the weight then increases by 
$d_i-d_j$. Repeating this process until we reach $\d_2$, we obtain the desired result. 
\end{proof}
The weight of difference sequences in $S(D)$ is not  compatible with direct or reverse  lexicographic order on $D$ but it is compatible with a partial order obtained from the transpositions of adjacent differences $d_i,d_j$ when $d_i> d_j$.
Here is a first result in that direction in the simple case where the difference set has only one number different from $1$:
\begin{proposition}\label{th5}
Let $d$  and $r$ be positive integers. The sets of partitions that are $r$-hooks with hook types corresponding to a difference set of the form $\{ d,1^{r-1}\}$ have hook types of the form 
\begin{equation}
\label{eq6}
\underbrace{(2s+d+2(r-s-1)\ldots,2s+d+2,2s+d}_{r-s},\underbrace{2s-1\ldots,3,1}_s)
\end{equation}
where $0\leq s<r$.
\end{proposition}
\begin{proof} The only possible factorization of the number $d$ as a product of $r$ integers taken from the difference set $\{ d,1^{r-1}\}$  is  $d=d\cdot 1^{r-1}$ and 
the hook type in equation (\ref{eq6}) has a difference sequence equal to $(1^{r-s-1},d,1^s)$ as in figure \ref{fig8}. The Ferrers diagram of a partition with difference set equal to $\{ d,1^{r-1}\}$ is the juxtaposition of a durfee square of size $s$ and two rectangles of sizes $(r-s)\times d_1$ and $(r-s)\times d_2$ with $d_1+d_2=d-1$: 
\begin{figure}[htbp]
\begin{center}
\includegraphics[width=.16 \textwidth]{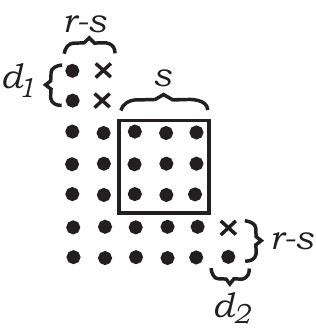}
\caption{A $5$-hook with difference sequence $(1,4,1^3)$}
\label{fig8}
\end{center}
\end{figure}

Thus a $r$-hook  with difference sequence  $(1^{r-s-1},d,1^s)$ is  a partition of weight $r^2+(r-s)(d-1)$ and the hook shape in (\ref{eq6}) is immediate.
\end{proof}
Now how do we recognize or construct a partition from its difference set  $\{ d_1,d_2,\ldots, d_t,1^{r-t}\} $ ?  First we have to choose integers $s_0\geq 0,s_1\geq 1,\ldots , s_t\geq 1$ such that $\sum_i s_i=r$. Then we start with a square of size $s_0\times s_0$. On top left corner of the square, we add a rectangle  with $d_1-1$ rows of lengths $\sum_{i=1}^rs_i$. This rectangle can slide around the square $s_0\times s_0$ so that one part of it comes out  on the bottom right corner of the square.  Then on top left corner of the previous rectangle, we add a second rectangle  with $d_2-1$ rows of lengths  $\sum_{i=2}^ts_i$ allowing again  a sliding around that respects the boundary of the previous rectangle. We continue the process of adding rectangles in this manner until the last one which has $d_t-1$ rows of lengths  $s_t$. Any Ferrers diagram that satisfies these conditions and only those have difference set $\{ d_1,d_2,\ldots, d_t,1^{r-t}\}$. 

As an example, we show in figure \ref{fig9} a $7$-hook with difference set $\{ 5,4,1^5\}$ where  $s_0=2,s_1=2,s_2=3$. 
\begin{figure}[htbp]
\begin{center}
\includegraphics[width=.25 \textwidth]{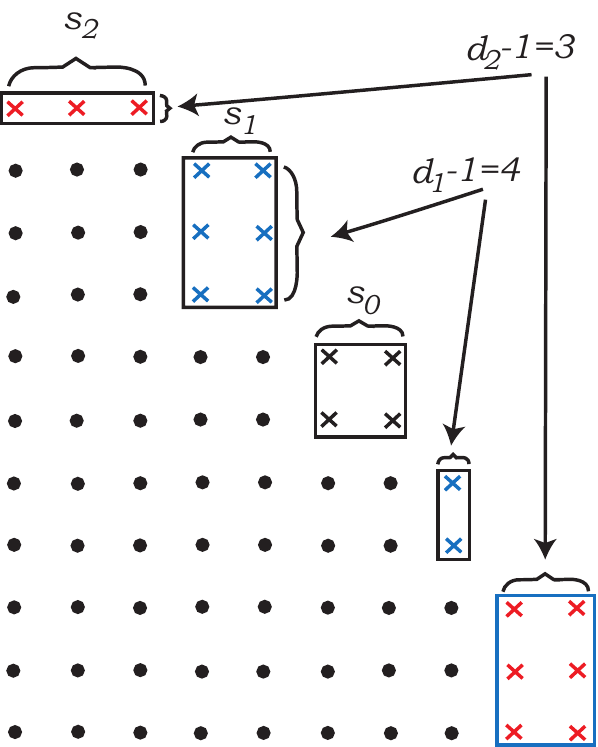}
\caption{A $7$-hook with $d_1=5,\, d_2=3$}
\label{fig9}
\end{center}
\end{figure}\\
So if we want to produce another set of partitions with product of differences equal to $15$, we can either permute the $d_i$, add ones in the difference set, or consider another factorisation of $15$ which is $15$ itself.  The same argument is valid in the general case. 

Of course these sets of partitions do not all have the same weight. Actually, when the $d_i$ are known and the $s_i$ are chosen, their weight is $r^2+\sum_{i=1}^t\left[ (d_i-1)\times \sum_{j=i}^ts_j \right]$. 


\section{More questions}\label{sec7}
There are several  questions that arise from this introductory study of the product of partitions $\l*\m$.  First, how could we export and extend the products $\l*\m$ and $\d_1\circ \d_2$ ? Partitions serve as index set for several families of objects in representation theory such as symmetric functions, irreducible characters and conjugacy classes of the symmetric group $S_n$. The product of partitions defined here induces a product on these family of objects worth investigating.

Moreover, integer partitions have  natural two and three dimensional  extensions. Is it possible to extend the product $\l*\m$ to extensions such as polyominos and plane partitions?  These questions are open for the moment.   


{\bf Acknowledgement} Many thanks to Christophe Reutenauer for helpful discussions and for pointing the 
connection between the monoids  $(P,*)$,  $(P/\sim,*)$ and triangular matrices. For more details on this subject see \cite{BR}. Thanks also to Adriano Garsia for the interesting discussions on this subject. 

\end{document}